\documentclass[letterpaper, 10 pt, journal, twoside]{ieeetran}
\IEEEoverridecommandlockouts
\usepackage{cite}
\usepackage[dvipsnames]{xcolor}
\usepackage{amsmath,amssymb,amsfonts,mathtools}
\usepackage{algorithmic}
\usepackage{graphicx}
\usepackage[export]{adjustbox}
\usepackage{tikz}
\usepackage{textcomp}
\usepackage[font=small]{caption}
\def\BibTeX{{\rm B\kern-.05em{\sc i\kern-.025em b}\kern-.08em
		T\kern-.1667em\lower.7ex\hbox{E}\kern-.125emX}}

\usepackage{amsthm}
\usepackage{mathrsfs}
\usepackage{hyperref}

\newcommand{\reals}{\mathbb{R}}
\newcommand{\transpose}{^\textrm{T}}

\DeclareMathOperator*{\argmin}{arg\,min}

\newtheorem{theorem}{Theorem}
\newtheorem{assumption}{Assumption}
\newtheorem{lemma}{Lemma}
\newtheorem{definition}{Definition}
\newtheorem{corollary}{Corollary}

\newtheorem{problem}{Problem}

\usepackage{verbatim}
\usepackage{url}
\usepackage{hyperref}
\usepackage{placeins}
\usepackage{soul}
\usepackage{tikz}

\newcommand{\dtime}[0]{\sigma}

\definecolor{lcolor}{rgb}{0,0,0.6}
\definecolor{edits}{rgb}{0,0,0}
\hypersetup{
	colorlinks=true,
	linkcolor=lcolor,
	filecolor=magenta,      
	urlcolor=black,
}

\newcommand{\regularversion}[1]{\iffalse{}#1\fi}
 \newcommand{\extendedversion}[1]{{#1}}

\begin{document}

%\title{\LARGE\bf Safety-Critical Online Control of a Spacecraft with Impulsive Actuators and Dwell Time Constraints}

\title{Safety-Critical Control for Systems with Impulsive Actuators and Dwell Time Constraints}%
% \title{Safety-Critical Control of a Spacecraft with Impulsive Actuators and Dwell Time Constraints}%
%\title{Stabilization and Safety of a Spacecraft with Impulsive Actuators and Dwell Time Constraints}%
% \title{Impulsive Control Barrier Functions for Spacecraft Control under Dwell Time Constraints}%

\author{Joseph Breeden and Dimitra Panagou\thanks{The authors are with the Department of Aerospace Engineering and Department of Robotics, University of Michigan, Ann Arbor, MI, USA. Email: \texttt{\{jbreeden,dpanagou\}@umich.edu}. The authors thank the Fran\c{c}ois Xavier Bagnoud Foundation for supporting this research.}}%

\maketitle

\thispagestyle{empty}

\begin{abstract}
This paper presents extensions of control barrier function (CBF) and control Lyapunov function (CLF) theory to systems wherein all actuators cause impulsive changes to the state trajectory, and can only be used again after a minimum dwell time has elapsed. 
These rules define a hybrid system, wherein the controller must at each control cycle choose whether to remain on the current state flow or to jump to a new trajectory. 
%The time between jumps is aperiodic and lower bounded.
We first derive a sufficient condition to render a specified %safe 
set forward invariant using extensions of CBF theory. We then derive related conditions to ensure asymptotic stability in such systems, and apply both conditions online in an optimization-based control law with aperiodic impulses. We simulate both results on a spacecraft docking problem with multiple obstacles.%
\end{abstract}

\vspace{-5pt}

\vspace{-9pt}
\section{Introduction}

Control Barrier Functions (CBFs) \cite{CBFs_Tutorial} are a tool for designing control laws that render state trajectories always inside a specified set. Each CBF converts a set of allowable states, herein called the \emph{CBF set}, to a set of allowable control inputs at every state in that set \cite{compatibility_checking}. Any control input within this set will render the future state trajectory inside the CBF set. The controller thus has freedom to work towards other goals, such as convergence, as long as the control remains within the input set generated by the CBF. CBFs thus provide a computationally tractable solution to many nonlinear constrained control problems. While the original formulations of CBFs \cite{CBFs_FirstUse,BarrierCertificates,AmesJournal} considered continuous-time systems, subsequent authors have published numerous extensions to sampled systems \cite{LCSS,CBFs_Mechanical,self_triggered_cbf,self_event_triggered,event_triggered_journal,AIAA}, discrete-time systems \cite{Discrete_CBFs,Discrete_MPC_CBFs}, and hybrid systems \cite{hybrid_synergistic_cbfs,hybrid_robust_part_ii,hybrid_cbfs_hybrid_systems,hybrid_impulse_inclusion}, among others. In this paper, we develop set invariance rules for a specific class of hybrid systems: systems with impulsive actuators that are only permitted to be used after a minimum \emph{dwell time} has elapsed since their previous use. This models, for instance, a spacecraft with chemical thrusters.

Impulsive systems are a special class of hybrid systems, and there has been much work on stability of hybrid systems over the past two decades \cite{impulsive_iss,impulsive_optimal,impulsive_biped,impulsive_event_triggered_stability,impulsive_event_triggered_iss,impulsive_event_triggered_consensus}, and more recently work on set invariance \cite{hybrid_impulse_inclusion,BarrierCertificates,hybrid_robust_part_i} and CBFs \cite{hybrid_cbfs_hybrid_systems,hybrid_robust_part_ii,hybrid_synergistic_cbfs,hybrid_robot_eye_set_avoidance,adding_obstacles} for hybrid systems. A hybrid system is a combination of a set of time intervals where a system \emph{flows} according to a state differential equation called the \emph{flow map}, and a set of times where the state \emph{jumps} (changes instantaneously) according to an algebraic function called the \emph{jump map}. Control may be applied along the flows, at the jumps (also called impulses), or both. In this letter, we study systems where control occurs only via jumps, and jumps occur only when control is applied, as is formalized in Section~\ref{sec:prelims}. 
% are either overly conservative \cite{impulsive_iss} or no longer directly applicable.
% Ignoring {\color{edits}the minimum dwell time} constraint, 
{\color{edits}The work in} \cite{hybrid_impulse_inclusion,hybrid_robust_part_i} show that a hybrid system renders a set forward invariant if 1) the flow map always lies within the tangent cone of the set, and 2) the image of the set through the jump map is a subset of the set. 
The authors in \cite{BarrierCertificates,hybrid_robust_part_ii,hybrid_cbfs_hybrid_systems} then rewrite these conditions for {\color{edits}CBFs and CBF sets. However, these two conditions
have no way to incorporate a minimum dwell time constraint (equivalently, a minimum time between events \cite{impulsive_event_triggered_iss}).} %, or are overly conservative due to a reliance on bounding the flow dynamics by exponential rates of divergence, thus requiring excessive corrections during the impulses \cite{impulsive_iss}.}
% However, these conditions are impractical when the minimum dwell time constraint is added.
%Ignoring this constraint, \cite{BarrierCertificates,hybrid_impulse_inclusion,hybrid_robust_part_i} provide sufficient conditions analogous to \cite{Nagumo1942} for forward invariance of general sets along a hybrid system, and \cite{hybrid_robust_part_ii,hybrid_cbfs_hybrid_systems} encode these conditions as barrier functions.

% Recall that existence of a CBF implies existence of a control input that renders the CBF set forward invariant \cite{AmesJournal}. Thus, 
Recall that the problem of finding a CBF is equivalent to the problem of finding a controlled-invariant set \cite{ACC2023}. For hybrid systems, this equivalency follows from, e.g., \cite[Def.~3.6]{hybrid_robust_part_ii}
%{\color{edits},}
and 
\cite[Def.~5]{adding_obstacles}. In this letter, due to the minimum dwell time constraint, rather than applying control to render the state %always
inside such a controlled-invariant set, we must apply control to render the state into a set whose forward reachable set remains a subset of the CBF set at least until the dwell time has elapsed. This is an inherently different problem than that addressed by typical CBFs \cite{AmesJournal,CBFs_Tutorial} or by the hybrid CBFs in \cite{hybrid_impulse_inclusion,BarrierCertificates,hybrid_robust_part_i,hybrid_cbfs_hybrid_systems,hybrid_synergistic_cbfs,hybrid_robust_part_ii,hybrid_robot_eye_set_avoidance,adding_obstacles}, and has more in common with {\color{edits}margins for ensuring set invariance between samples under} sampled controllers such as \cite{sampled_data_2022,sampled_data_unknown,CBFs_Mechanical,LCSS,self_triggered_cbf,self_event_triggered}. 
% {\color{edits}These sampling margins, modified for impulsive control rather than zero-order-hold control, are the basis of our CBF approach. 
{\color{edits}This paper applies the same concept of sampling margins as in \cite{LCSS}, now modified for impulsive rather than zero-order-hold control, to guarantee set invariance under a minimum dwell time.
Additionally, in Section~\ref{sec:mpc}, we propose a variation of our method for reducing conservatism}. %, and we compare both methods in simulation in Section~\ref{sec:simulations}.} %for use with much longer} That said, we are also interested in much longer dwell times than were present in \cite{CBFs_Mechanical,LCSS,sampled_data_2022,sampled_data_unknown}, and {\color{edits}propose a strategy to reduce the conservatism of these papers}
% address this challenge 
% in Section~\ref{sec:mpc}.

Finally, the addition of the minimum dwell time constraint also complicates stability. The work in \cite{impulsive_iss} provides a {\color{edits}formula for a} maximum dwell time at which stability is still guaranteed{\color{edits}, and} similar {\color{edits}stability certificates for specified dwell times} are presented in \cite{impulsive_event_triggered_iss,impulsive_event_triggered_stability,impulsive_event_triggered_consensus}. {\color{edits}However, all of these results} are overly restrictive, because they all place weak assumptions {\color{edits}(e.g., exponentially bounded divergence)} on the flows in exchange for strong {\color{edits}requirements} {\color{edits}(e.g., rapid exponential contractivity)} on the jumps. This is sensible in general, since the jumps are controlled and the flows are uncontrolled, but the spacecraft community has long developed controllers with weaker assumptions on the jumps \cite{First_APF_Rendezvous,Rendezvous_APF_Constraints,rendezvous_APF_conference}, though not with the desired minimum dwell time. Thus, building up from the minimum dwell time constraint, this letter presents conditions to 
\vspace{-3pt}
\begin{enumerate}
	\item render a CBF set forward invariant subject to impulsive control {\color{edits}with a minimum dwell time constraint}, and
	\item render the origin asymptotically stable subject to the same impulsive control and dwell time rules.
\end{enumerate}
\vspace{-3pt}
{{\color{edits}This} paper is organized as follows. Section~\ref{sec:prelims} presents the system model. Section~\ref{sec:methods} presents the set invariance strategy, the asymptotic stability strategy, and some mathematical tools. Section~\ref{sec:simulations} presents {\color{edits}simulations of these methods on a satellite docking problem}. Section~\ref{sec:conclusions} presents concluding remarks.}

\section{Preliminaries} \label{sec:prelims}

% \subsection{Notation}

\noindent\textbf{Notations:}
Given a time domain $\mathcal{T} \subseteq\reals$, spatial domain $\mathcal{X}\subseteq\reals^n$, and function $\color{edits}\eta\color{black}:\mathcal{T}\times\mathcal{X}\rightarrow\reals$, denoted $\color{edits}\eta\color{black}(t,x)$, let $\partial_t \color{edits}\eta\color{black}$ denote the partial derivative with respect to $t$. Let $\nabla \color{edits}\eta\color{black}$ denote the gradient row vector with respect to $x$. Let $\dot{\color{edits}\eta} = \partial_t \color{edits}\eta\color{black} + \nabla \color{edits}\eta\color{black} \dot{x}$ denote the total derivative of ${\color{edits}\eta}$ in time. Let $\mathbb{N}$ denote the set of nonnegative integers. Let $\| \cdot \|$ denote the \regularversion{2-}\extendedversion{Euclidean }norm. A continuous function $\alpha:\reals_{\geq0}\rightarrow\reals_{\geq 0}$ \regularversion{\color{edits}is}\extendedversion{belongs to} class-$\mathcal{K}_\infty$, denoted $\alpha\in\mathcal{K}_\infty$ if 1) $\alpha(0) = 0$, 2) $\alpha$ is strictly increasing, and 3) $\lim_{\lambda\rightarrow\infty}\alpha(\lambda) = \infty$. Let $\mathcal{K}_r$ denote the set of continuous functions $\beta:\reals_{\geq 0}\rightarrow\reals_{\geq 0}$ satisfying 1) $\beta(0) = 0$,~and~2)~$\beta(\lambda) > 0$~for~all~$\lambda > 0$.
%there exists some $a\in\reals_{>0}$ such that $\alpha$ is strictly increasing on $[0,a)$, and 3) $\beta$ is nondecreasing on its entire domain. \todo{can this be replaced with $\beta(\lambda) > 0$ for all $\lambda > 0$?}

% \subsubsection{Model}

\vspace{3pt}
\noindent\textbf{Model:}
Spacecraft with chemical thrusters are frequently modeled as evolving according to an ordinary differential equation (ODE) with impulsive jumps. When activated, the thruster subsystem causes an instantaneous change, called an impulse, to the spacecraft velocity, and then the system flows according to the ODE until the next impulse is applied. Control {\color{edits}can} only {\color{edits}be} applied at the impulses\extendedversion{, as we assume that there are no other actuators capable of applying control during the flows}. We also assume the following two restrictions on impulses:
\renewcommand{\theenumi}{R-\arabic{enumi}}
\begin{enumerate}
	\item the controller is sampled with fixed period $\Delta t$ and an impulse can only be applied at the sample times; and \label{req:dt}
	\item the controller can only apply an impulse at least $\Delta T$ after the last impulse was applied, where $\Delta T > \Delta t$. \label{req:dT}
\end{enumerate}

Let $\mathcal{T}\subseteq\reals$ be a \regularversion{\color{edits}time domain}\extendedversion{set of considered times}, $\mathcal{X}\subseteq\reals^n$ be the state space, and $\mathcal{U}\subseteq\reals^m$ be the set of allowable controls. To encode \ref{req:dt}, let
\begin{subequations}\begin{equation}
	\mathcal{D}_0 \triangleq \{t\in\mathcal{T} \mid t = t_0 + k \Delta t, k \in \mathbb{N}\} \label{eq:def_D0}
\end{equation}
be the set of controller sample times originating from an initial time $t_0\in\mathcal{T}$.
To encode \ref{req:dT}, let the additional state $\dtime \in \reals_{\geq 0}$ encode the time since the last impulse was applied. A tuple $(t,\sigma)$ is an \emph{impulse opportunity} if $t\in\mathcal{D}_0$ and $\sigma \geq \Delta T$, or equivalently, if $(t,\sigma)$ lies in the set of impulse opportunities
\begin{equation}
	\mathcal{D} \triangleq \mathcal{D}_0 \times \{ \dtime \in \reals_{\geq 0} \mid \dtime \geq \Delta T\} \label{eq:def_D} \,. \\
\end{equation}
The control is thus a map $u:\mathcal{D}\times\mathcal{X}\rightarrow\mathcal{U}$ defined only at the set of impulse opportunities $\mathcal{D}$. The time $\Delta T$ is called the minimum \emph{dwell time} between impulses \cite{impulsive_iss}. {\color{edits}Assume that $\Delta T = q \Delta t$ for some $q\in\mathbb{N}$.}
% Denote the set of all other $(t,\sigma)$ combinations as
% \begin{equation} 
% 	\mathcal{C} \triangleq (\mathcal{T}\times\reals_{\geq 0}) \setminus \mathcal{D} \label{eq:def_C} \,,
% \end{equation}

We can thus model the spacecraft generally as
\begin{equation}
	\begin{cases}
		\;\begin{cases}
			\dot{x} = f(t,x) \\
			\dot{\sigma} = 1
		\end{cases} & (t,\dtime) \notin \mathcal{D} \\
		\;\begin{cases}
			x^+ = g(t,x,u) \\
			\sigma^+ = \sigma \textrm{ if } u=0 \\
			\sigma^+ = 0 \textrm{ if } u \neq 0
		\end{cases} & (t,\dtime) \in \mathcal{D}
	\end{cases} \label{eq:xdot}
\end{equation}%
\label{eq:model}%
\end{subequations}
The system \eqref{eq:xdot} defines a hybrid system with flow set $\mathcal{C}\triangleq (\mathcal{T}\times\reals_{\geq 0}) \setminus \mathcal{D}$, flow map $f:\mathcal{T}\times\mathcal{X}\rightarrow\reals^n$, jump set $\mathcal{D}$, and jump map $g:\mathcal{T}\times\mathcal{X}\times\mathcal{U}\rightarrow\mathcal{X}$. We note that \eqref{eq:xdot} has time-dependent jumps, and therefore is also a timed automaton \cite{time_automata}. In this paper, we assume that the maps $f$ and $g$ are {\color{edits}known and} single-valued (rather than being differential inclusions), {\color{edits}that $g(t,x,0) = x$ for all $t\in\mathcal{T},x\in\mathcal{X}$,} and that solutions to \eqref{eq:model} exist and are unique for all $t\in\mathcal{T}$.
Also assume that $\sigma(t_0) = \Delta T$ at the initial time $t_0$, so that the initial state tuple $(t_0,\sigma(t_0),x(t_0))$ is an impulse opportunity.

%\todo{should I define $z(t) = (t,\sigma(t),x(t))$ for compactness, or does that make the notation even harder to follow?}

% \noindent
% \makebox[\linewidth][s]{\indent Note that at every impulse opportunity $(t,\sigma)\in\mathcal{D}$, the}

% \noindent
Note that at every impulse opportunity $(t,\sigma)\in\mathcal{D}$, the
controller $u$ can choose whether or not to apply an impulse, so impulses will generally be aperiodic, and may lack an average dwell time as in \cite{impulsive_iss}. 
%The theorems in Section~\ref{sec:methods} thus establish conditions that the system should satisfy both if $u = 0$ and if $u\neq 0$.
For brevity in Section~\ref{sec:methods}, given a control law $u:\mathcal{D}\times\mathcal{X}\rightarrow\mathcal{U}$, denote the set of impulse opportunities where the control law chooses to not apply an impulse as
\begin{equation}
	\mathcal{Z}_\textrm{coast} \triangleq \{(t,\sigma,x)\in\mathcal{D}\times\mathcal{X} \mid u(t,\sigma,x) = 0\} \label{eq:Z_coast} \,.
\end{equation}

The central problem addressed in Section~\ref{sec:methods} is as follows.

\begin{problem} \label{problem}
	Given dynamics \eqref{eq:model} and a set $\mathcal{S}_\textrm{safe}(t)\subset\mathcal{X}$, derive conditions on the control $u$ that are sufficient to 1) guarantee $x(t)$ remains in $\mathcal{S}_\textrm{safe}(t), \forall t\in\mathcal{T}$, %{\color{edits}including between control impulses,} 
    and 2) render the origin asymptotically stable, where we assume $0\in\mathcal{S}_\textrm{safe}(t), \forall t\in\mathcal{T}$.
\end{problem}

% The conditions arising from Problem~\ref{problem} can then be enforced online using optimization-based control laws as is typical in the CBF literature \cite[Sec.~II-C]{CBFs_Tutorial}. Note how unlike \cite{CBFs_Tutorial}, in this paper, we will allow these optimizations to be nonlinear programs (e.g. see \eqref{eq:docking_control}). Such optimizations are more computationally expensive to solve than the quadratic programs in \cite{CBFs_Tutorial}, but we assume that this cost is acceptable because of the long dwell time $\Delta T$ between impulses.

The conditions arising from Problem~\ref{problem} can then be enforced online using optimization-based control laws as is typical in the CBF literature \cite[Sec.~II-C]{CBFs_Tutorial}. Unlike \cite{CBFs_Tutorial}, in this letter, we allow these optimizations to be nonlinear programs. Such programs are more computationally expensive than the quadratic programs in \cite{CBFs_Tutorial}, but we assume that this cost is acceptable because of the long dwell time $\Delta T$ between impulses.

%\subsection{Safety and Problem Formulation}
%
%Suppose we are given a nonempty set $\mathcal{S}_\textrm{safe}(t) \subset \mathcal{X}$, potentially time-varying, called the \emph{safe set}. The principal requirement of the control $u$ in \eqref{eq:model} is that given any $x(t_0)\in\mathcal{S}_\textrm{safe}(t_0)$, the control law must render the state $x(t)$ always inside $\mathcal{S}(t)$ for all $t\geq t_0,t\in\mathcal{T}$. A control law $u:\mathcal{D}\times\mathcal{X}\rightarrow\mathcal{U}$ that accomplishes this is called \emph{safe}. A control law that causes $\lim_{t\rightarrow\infty} x(t) = 0$ is called \emph{convergent}. The central question of this letter is how to develop safe control laws subject to the dynamics \eqref{eq:model}.
%\color{black}

%The central question of this letter is how to develop Given the system and timing constraints \eqref{eq:model}, the central question of this letter is how to develop convergent and provably safe impulsive control laws.

\section{Impulsive Timed Control Barrier Functions and Control Lyapunov Functions} \label{sec:methods}

%To employ the method of CBFs \cite{CBFs_Tutorial}, we further assume that there exists a finite number $n_\textrm{cbf}$ of continuous functions $h_i:\mathcal{T}\times\mathcal{X}\rightarrow\reals$ such that $\mathcal{S}_\textrm{safe}(t) \subseteq \cap_{i=1}^{n_\textrm{cbf}} \{ x\in\mathcal{X} \mid h_i(t,x) \leq 0 \}$ for all $t\in\mathcal{T}$. In Section~\ref{sec:cbfs}, we present conditions for each $h_i$ to be an ``impulsive timed CBF'' (ITCBF), and by consequence to be useful for safety. In Section~\ref{sec:control}, we briefly discuss control laws using ITCBFs. In Section~\ref{sec:stable}, we derive stability and convergence conditions specialized to the model \eqref{eq:model} and to common spacecraft dynamics. Finally, in Section~\ref{sec:bounds}, we discuss in more detail methods of developing the bounding functions required to define ITCBFs.

In this section, we first present some definitions and tools in Section~\ref{sec:bounds}, before using these tools to address invariance of a subset of $\mathcal{S}_\textrm{safe}(t)$ in Section~\ref{sec:cbfs}. We then address stability of the origin in two parts in Sections~\ref{sec:one-step}-\ref{sec:stability}, and provide examples and additional tools in Section~\ref{sec:examples}.

\vspace{-4pt}
\subsection{Flows and Bounding Functions} \label{sec:bounds}

In this letter, we will utilize predictions about the future state. Suppose that no jumps occur in some interval $[t,\tau]\subset\mathcal{T}$. Then define the flow operator $p:\mathcal{T}\times\mathcal{T}\times\mathcal{X}\rightarrow\mathcal{X}$ as
\begin{equation}
	p(\tau, t, x) = y(\tau) \textrm{ where } \dot{y}(s) = f(s,y(s)),\; y(t) = x\,.
	\label{eq:state_flow}
\end{equation}

Next, we are interested in approximations of the future state. Given a scalar function $h:\mathcal{T}\times\mathcal{X}\rightarrow\reals$, and an initial state $(t,x)$, denote by $\psi_h:\mathcal{T}\times\mathcal{T}\times\mathcal{X}\rightarrow\reals$ any function satisfying
\begin{equation}
	\psi_h(\tau, t, x) \geq h(s,p(s,t,x)), \;\forall s \in [t,\tau] \,. \label{eq:upper_bound}
\end{equation}
That is, $\psi_h$ is an upper bound on the evolution of the function $h$ for any interval $[t,\tau]$ during which there are no control impulses. Methods to find such a bounding function are described in \cite{LCSS,AIAA,CBFs_Mechanical,sampled_data_2022,sampled_data_unknown,self_triggered_cbf} and others, and thus are only briefly elaborated upon here in Section~\ref{sec:examples}. We note that \cite{LCSS,AIAA,CBFs_Mechanical,sampled_data_2022,sampled_data_unknown,self_triggered_cbf} all include a term that accounts for the effects of the control input $u\in\mathcal{U}$, whereas this term can be ignored here since $f$ in \eqref{eq:model} is independent of $u$.

\vspace{-4pt}
\subsection{Set Invariance} \label{sec:cbfs}

We first address the safety part of Problem~\ref{problem}. 
To apply the method of CBFs, we seek a function $h:\mathcal{T}\times\mathcal{X}\rightarrow\reals$ such that the set 
%Assume that $\mathcal{S}_\textrm{safe}(t)$ in Problem~\ref{problem} can be described by the intersection of the sublevel sets of finitely many functions $h_i:\mathcal{T}\times\mathcal{X}\rightarrow\reals$, so that we can apply the method of CBFs. Going forward, let $h:\mathcal{T}\times\mathcal{X}\rightarrow\reals$ be a single candidate CBF and denote the zero sublevel of $h$ set as
\vspace{-5pt}
\begin{equation}
	\mathcal{S}_h(t) \triangleq \{ x\in \mathcal{X} \mid h(t,x) \leq 0\} \label{eq:sh}
\end{equation}
satisfies $\mathcal{S}_h(t) \subseteq\mathcal{S}_\textrm{safe}(t), \forall t\in\mathcal{T}$. % and $\mathcal{S}_h$ is controlled invariant.
The definition of CBF \cite[Def. 5]{AmesJournal} can be generalized to 
the system \eqref{eq:model} as follows.
% systems of the form \eqref{eq:model} as follows.

\begin{definition}\label{def:itcbf}
	Let $\psi_h$ be as in \eqref{eq:upper_bound}. A continuous function $h:\mathcal{T}\times\mathcal{X}\rightarrow\reals$ is an \emph{Impulsive Timed Control Barrier Function (ITCBF)} for the system \eqref{eq:model} if 
	\begin{equation}
		\inf_{u\in\mathcal{U}} \psi_h(t+\Delta T, t, g(t,x,u)) \leq 0, \forall x\in\mathcal{S}_h(t),\forall t\in\mathcal{T} \,. \label{eq:cbf_definition}
	\end{equation}
\end{definition}

Note that 1) we relax \cite[Def. 5]{AmesJournal} to no longer require differentiability of $h$, though differentiability is helpful when applying tools from \cite{LCSS,AIAA,CBFs_Mechanical,sampled_data_2022,sampled_data_unknown,self_triggered_cbf}, and 2) condition \eqref{eq:cbf_definition} does not include a class-$\mathcal{K}$ function, as this is unnecessary in sampled controllers such as \eqref{eq:model}. 
\regularversion{The following theorem then provides sufficient conditions for forward invariance of $\mathcal{S}_h(t)$.}

\begin{theorem} \label{thm:safety}
	Given an ITCBF $h:\mathcal{T}\times\mathcal{X}\rightarrow\reals$ for the system \eqref{eq:model}, let $\mathcal{S}_h(t)$ be as in \eqref{eq:sh}, and $\psi_h$ as in \eqref{eq:upper_bound}. Let $u:\mathcal{D}\times\mathcal{X}\rightarrow\mathcal{U}$ be a control law, and let $\mathcal{Z}_\textrm{coast}$ be as in \eqref{eq:Z_coast}. If $u$ satisfies
	\begin{subequations}\begin{align}
			\hspace{-6pt} \psi_h(t+\Delta t, t, x) \leq 0, \;\;\; &\forall (t,\sigma,x)\in(\mathcal{D}\times\mathcal{S}_h)\cap\mathcal{Z}_\textrm{coast} \label{eq:cbf_cond1} , \hspace{-6pt}\\
			\hspace{-6pt} \psi_h(t+\Delta T, t,y) \leq 0, \;\;\; &\forall (t,\sigma,x)\in(\mathcal{D}\times\mathcal{S}_h)\setminus\mathcal{Z}_\textrm{coast} \label{eq:cbf_cond2} , \hspace{-6pt}
		\end{align}\label{eq:cbf_condition_relaxed}\end{subequations}
	where $y = g(t,x,u(t,\dtime,x))$, then $u$ renders time-varying set $\mathcal{S}_h(t)$ forward invariant for all $t\in\mathcal{T}$.
\end{theorem}
\begin{proof}
	Given $(t_0,\sigma(t_0))\in\mathcal{D}$ and $x(t_0)\in\mathcal{S}_h(t_0)$, divide $\{t\in\mathcal{T} \mid t \geq t_0\}$ into a sequence of intervals $\mathcal{I}_k = [t_k, t_{k+1}]$, $k\in\mathbb{N}$, where $t_{k+1}  > t_k$. Then a sufficient condition for $u$ to render $\mathcal{S}_h(t)$ forward invariant for all future $t\in\mathcal{T}$ is for the following two properties to hold for every $k\in\mathbb{N}$: 1) $u$ renders $x(t)\in\mathcal{S}_h(t)$ for all $t$ in the interval $\mathcal{I}_k$, and 2) the endpoint $t_{k+1}$ of $\mathcal{I}_k$ is an impulse opportunity. If $u=0$, then condition \eqref{eq:cbf_cond1} implies that both properties hold for $t_{k+1} = t_k+\Delta t$. If $u\neq0$, then condition \eqref{eq:cbf_cond2} implies that both properties hold for $t_{k+1} = t_k + \Delta T$. %Since either \eqref{eq:cbf_cond1} or \eqref{eq:cbf_cond2} must hold at every $k\in\mathbb{N}$ 
 %cover all of $\mathcal{S}_h(t_k)$ at every impulse opportunity $(t_k,\sigma(t_k))\in\mathcal{D}$, it follows that the above properties hold for every $k\in\mathbb{N}$, so 
    Thus, $u$ renders $\mathcal{S}_h(t)$ forward invariant.
\end{proof}

Thus, we have two conditions analogous to \cite[Cor.~2]{AmesJournal} that render sets of the form \eqref{eq:sh} forward invariant subject to %impulsive dynamics with a minimum delay between control impulses as in \eqref{eq:model}. 
the impulsive dynamics \eqref{eq:model}.
The remaining challenge is to determine functions $h$ and $\psi_h$ satisfying \eqref{eq:cbf_definition} and \eqref{eq:upper_bound}, respectively. We first discuss conditions for asymptotic stability before providing examples of $h$ and $\psi_h$ in Section~\ref{sec:examples}. 

\subsection{One-Step MPC Impulsive Stability} \label{sec:one-step}

We now begin to address the stability part of Problem~\ref{problem}. There has been much work on stability of hybrid systems with continuous actuators \cite{impulsive_biped,hybrid_zero_dynamics,impulsive_optimal,self_event_triggered}, impulsive actuators \cite{impulsive_event_triggered_consensus,impulsive_event_triggered_stability,impulsive_event_triggered_iss}, or both \cite{impulsive_iss}. 
In summary, given a Lyapunov function $V:\mathcal{T}\times\mathcal{X}\rightarrow\reals_{\geq 0}$, the conditions \cite[Eq.~5]{impulsive_event_triggered_stability}, \cite[Eq.~8]{impulsive_event_triggered_iss}, and \cite[Eq.~4b]{impulsive_iss} state that if $V(t,g(t,x,u)) \leq c V(t,x)$ for $c\in(0,1)$, then for sufficiently frequent jumps, the origin of the system \eqref{eq:xdot} is exponentially stable. These conditions can be readily applied to stabilize \eqref{eq:model} using periodic impulses. However, when the dwell time $\Delta T$ is large, a more efficient strategy may be to examine the predicted value of the Lyapunov function after $\Delta T$ has elapsed rather than immediately after the impulse. To this end, consider the following lemma.

\begin{assumption}\label{as:lyap}
	Let $V:\mathcal{T}\times\mathcal{X}\rightarrow\reals_{\geq 0}$ be a continuously differentiable function satisfying
	\begin{equation}
		\alpha_1(\|x\|) \leq V(t,x) \leq \alpha_2(\|x\|) \label{eq:lyapunov_function}
	\end{equation}
	for all $x\in\mathcal{X}$ and all $t\in\mathcal{T}$ for two functions $\alpha_1,\alpha_2\in\mathcal{K}_\infty$.
\end{assumption}

\begin{lemma}\label{lemma:one-step_mpc}
	Let Assumption~\ref{as:lyap} hold. Assume that there exists $\alpha_3\in\mathcal{K}_\infty$ such that $f$ in \eqref{eq:model} satisfies $\|f(t,x)\| \leq \alpha_3(\|x\|)$ for all $t\in\mathcal{T}$ and $x\in\mathcal{X}$. Let $p$ be as in \eqref{eq:state_flow}. Let $u:\mathcal{D}\times\mathcal{X}\rightarrow\mathcal{U}$ be a control law, and denote $\mathcal{Z}_1\equiv\mathcal{Z}_\textrm{coast}$ as in \eqref{eq:Z_coast} and $\mathcal{Z}_2 = (\mathcal{D}\times\mathcal{X})\setminus \mathcal{Z}_1$. For the system \eqref{eq:model}, if $u$ satisfies
	\begin{subequations}
		\begin{align}
%			u(t,\sigma,x) = 0 \iff x \in \mathcal{X}_1&(t) \,, \label{eq:mpc_cond_u0} \\
			%
			V(t+\Delta t, p(t+\Delta t, t, x)) \leq V(t,x), \;\; & \forall (\cdot) \in \mathcal{Z}_1 , \hspace{-2pt} \label{eq:mpc_cond_flow_outside} \\
			V({t+\Delta T}, p({t \hspace{-0.6pt}+\hspace{-0.6pt} \Delta T}, t, y)) \leq V(t,x), \;\; &\forall (\cdot)\in\mathcal{Z}_2 , \hspace{-2pt} \label{eq:mpc_cond_jump_outside}
		\end{align}\label{eq:mpc_conditions}\end{subequations}%
	where $(\cdot)=(t,\sigma,x)$ and $y = g(t,x,u(t,\sigma,x))$, then $u$ renders the origin uniformly stable as in \cite[Def.~4.4]{Khalil}.
\end{lemma}
\begin{proof}
	First, note that \eqref{eq:lyapunov_function} implies that every sublevel set of $V$ is compact and contains the origin. Also note that we assumed solutions to \eqref{eq:model} always exist, so $p$ in \eqref{eq:state_flow} and \eqref{eq:mpc_conditions} is well defined. 
  
    Let $(t_k,\sigma(t_k))\in\mathcal{D}$ be an impulse opportunity.
    For brevity, denote $z_k = (t_k, \sigma(t_k), x(t_k))$.
	First, suppose that $z_k\in\mathcal{Z}_2$, which implies that $u(z_k) \neq 0$. Then no other impulses can be applied for $t\in(t_k,t_k+\Delta T)$, so $x(t_k+\Delta T) = p(t_k+\Delta T, t_k, g(t_k, x(t_k), u(z_k)))$, and thus \eqref{eq:mpc_cond_jump_outside} ensures that $V(t_k+\Delta T, x(t_k+\Delta T))$ at the next impulse opportunity $(t_k+\Delta T,\sigma(t_k+\Delta T))\in\mathcal{D}$ is upper bounded by $V(t_k,x(t_k))$. Since $V$ is bounded at both $t_k$ and $t_k+\Delta T$ and $V$ satisfies \eqref{eq:lyapunov_function}, it follows that both $\|x(t_k)\|$ and $\|x(t_k + \Delta T)\|$ are also bounded. Since $\|f(t,x)\|$ is bounded by $\alpha_3(\|x\|)$ and $\Delta T$ is fixed and finite, the amount that $V$ can grow during the jump at $x(t_k^+)$ and during the flow $t\in(t_k,t_k+\Delta T)$ are bounded as $V(t,x(t)) \leq V(t_k,x(t_k)) + \alpha_4(\|x(t_k)\|)$ for some $\alpha_4 \in \mathcal{K}_\infty$. 
	 
	Next, suppose that $z_k\in\mathcal{Z}_1$. Then $u(z_k) = 0$, so the next impulse opportunity will occur at $t_k + \Delta t$. By \eqref{eq:mpc_cond_flow_outside}, $V(t_k+\Delta t,x(t_k+\Delta t)) \leq V(t_k, x(t_k))$. By the same argument as the prior case, $V$ is again bounded as $V(t,x(t)) \leq V(t_k,x(t_k)) + \alpha_4(\|x(t_k)\|)$ for some $\alpha_4 \in \mathcal{K}_\infty$ for all $t\in[t_k,t_k+\Delta t)$. % until the next impulse opportunity at $(t_k+\Delta t, \sigma(t_k+\Delta t)) \in \mathcal{D}$.
	
	Thus, $V$ is nonincreasing at the impulse opportunities $(t_k,\sigma(t_k))\in\mathcal{D}$, and the maximum value of $V(t,x(t)) - V(t_k,x(t_k))$ is uniformly bounded (i.e. uniform in $t$) by $\alpha_4(\|x(t_k)\|)$ for $t$ between the impulse opportunities. Thus, for initial state $x(t_0)\in\mathcal{X}$, $(t_0,\sigma(t_0))\in\mathcal{D}$, and any future time $t\in\mathcal{T}$, $t\geq t_0$, it holds that $\|x(t)\| \leq \alpha_1^{-1}(V(t,x(t))) \leq \alpha_1^{-1}(V(t_k,x(t_k)) + \alpha_4(\|x(t_k)\|)) \leq \alpha_1^{-1}(V(t_0,x(t_0)) + \alpha_4(\|x(t_0)\|))  \leq \alpha_1^{-1}(\alpha_2(\|x(t_0)\|) + \alpha_4(\|x(t_0)\|))$. This is equivalent to uniform stability of the origin.
\end{proof}

Lemma~\ref{lemma:one-step_mpc} differs from \cite{impulsive_event_triggered_stability,impulsive_iss,impulsive_event_triggered_iss} in three ways. First, \eqref{eq:mpc_conditions} provides conditions on the future state, which is explicitly computed using \eqref{eq:state_flow}, rather than the present state. Second, these predictions allow us to avoid explicitly checking for upper bounds on the growth of $V$ during flows, as is required in \cite{impulsive_event_triggered_stability,impulsive_event_triggered_iss}. Third, Lemma~\ref{lemma:one-step_mpc} allows for aperiodic impulses, as long as \eqref{eq:mpc_conditions} are checked at their respective frequencies.

We refer to \eqref{eq:mpc_conditions} as a ``one-step Model Predictive Control (MPC)'' strategy. That is, to evaluate \eqref{eq:mpc_conditions}, we input the control $u$ at a single (i.e. ``one-step'') time instance, make a prediction using \eqref{eq:state_flow}, and then check a condition on $V$, analogous to checking constraints in an MPC optimization.
%We can encode these conditions in optimization based controllers as in \cite[Eq.~32]{AmesJournal} (see also \todo{eq}).
Note that encoding \eqref{eq:mpc_conditions} into an optimization problem could be computationally expensive, since checking \eqref{eq:mpc_conditions} entails computing the solution to a differential equation during every iteration of the optimization. In Section~\ref{sec:simulations}, we assume that this cost is acceptable, %(which is reasonable because impulses are computed infrequently), 
or that we have an analytic form for the solution, as is the case for many spacecraft orbits. 
%The left hand sides of \eqref{eq:mpc_cond_jump_outside}-\eqref{eq:mpc_cond_flow_outside} could also be replaced with a bounding function $\psi_V$ as in \eqref{eq:upper_bound} to reduce computational requirements; we use this idea more extensively in the following subsection.
% Next, we propose stricter conditions for stability that instead make use of approximate predictions
% Next, we present a specialization of Lemma~\ref{lemma:one-step_mpc} that 
% potentially improves controller performance.
%uses approximation functions as in \eqref{eq:upper_bound} rather than explicitly computing $p$ in \eqref{eq:state_flow} and \eqref{eq:mpc_conditions}.

\vspace{-1pt}
\subsection{Impulsive Stability via Restriction to Stable Flows} \label{sec:stability}

% In this subsection, we present a specialization of Lemma~\ref{lemma:one-step_mpc}. % that can improve controller performance and reduce computational requirements by using approximation functions as in \eqref{eq:upper_bound}.
% In this subsection, we present a specialization of Lemma~\ref{lemma:one-step_mpc} that uses approximation functions as in \eqref{eq:upper_bound} rather than exactly computing \eqref{eq:state_flow}, and that potentially improves controller performance.
% focuses on the flows stability conditions that make use of approximate predictions of the state trajectory between impulse opportunities, rather than the exact predictions used in \eqref{eq:mpc_conditions}. 
% 
% Lemma~\ref{lemma:one-step_mpc} allows us to design controllers that ignore the effects of the flows in \eqref{eq:model}. Next, we will focus on designing controllers that explicitly make use of these flows, by placing stronger restrictions on the flow---restrictions that arise naturally in spacecraft dynamics---than in \cite{impulsive_event_triggered_stability,impulsive_iss,impulsive_event_triggered_iss}.
% 
Motivated by fuel efficiency, a strategy in aerospace systems (e.g. \cite{First_APF_Rendezvous}) is to allow a system to coast uncontrolled until a control impulse is necessary to continue stabilization. In this subsection, we implement this strategy subject to constraints \ref{req:dt}\regularversion{\color{edits},}\extendedversion{ and} \ref{req:dT} via a specialization of Lemma~\ref{lemma:one-step_mpc}. 
% This entails a greater focus on the flows between impulses rather than the value of $V$ after $t+\Delta T$. 
In technical terms, given a Lyapunov function $V$ as in \eqref{eq:lyapunov_function}, we seek to render the set 
\begin{equation}
	\mathcal{S}_v(t) \triangleq \{ x\in\mathcal{X} \mid v(t,x) \leq 0\} \label{eq:decreasing_set}
\end{equation}
forward invariant, where, for readability, we denote 
\begin{equation}
	v(t,x) \equiv \dot{V}(t,x) = \partial_t V(t,x) + \nabla V(t,x) f(t,x) \,. \label{eq:def_v}
\end{equation}
This is possible under dynamics \eqref{eq:model} if $v:\mathcal{T}\times\mathcal{X}\rightarrow\reals$ is also an ITCBF as in Definition~\ref{def:itcbf}. 
Let $\psi_v$ be an upper bound for $v$ analogous to $\psi_h$ in \eqref{eq:upper_bound}. 
% The idea of the following theorem is to provide sufficient conditions for stability of the origin using $\psi_v$ that serve as alternatives to the ``one-step MPC'' strategy in \eqref{eq:mpc_conditions}. Divide the state space into two sets: 1) $\mathcal{Z}_1\cup\mathcal{Z}_2$, where \eqref{eq:mpc_conditions} apply, and 2) $\mathcal{Z}_3\cup\mathcal{Z}_4$, where the new conditions \eqref{eq:conditions} apply.
In the following theorem, we provide new conditions to establish stability using such a coasting strategy. However, if $x(t_0) \notin \mathcal{S}_v(t_0)$, then these conditions will not initially apply, so we instead fall back on the ``one-step MPC'' strategy in \eqref{eq:mpc_conditions}. Divide the state space into two sets: 1) $\mathcal{Z}_1\cup\mathcal{Z}_2$, where the controller enforces \eqref{eq:mpc_conditions}, and 2) $\mathcal{Z}_3\cup\mathcal{Z}_4$, where the controller enforces the new conditions \eqref{eq:conditions}.

\begin{theorem} \label{thm:stable}
	Let Assumption~\ref{as:lyap} hold. Assume that there exists $\alpha_3\in\mathcal{K}_\infty$ such that $f$ in \eqref{eq:model} satisfies $\|f(t,x)\| \leq \alpha_3(\|x\|)$ for all $t\in\mathcal{T}$ and $x\in\mathcal{X}$. Let $v$ be as in \eqref{eq:def_v}, $\psi_v$ be as in \eqref{eq:upper_bound}, and $p$ be as in \eqref{eq:state_flow}.  
	Let $\mathcal{Z}_1$, $\mathcal{Z}_2$, $\mathcal{Z}_3$, and $\mathcal{Z}_4$ be four disjoint sets such that $\mathcal{Z}_1\cup\mathcal{Z}_3 = \mathcal{Z}_\textrm{coast}$ in \eqref{eq:Z_coast}, and $\mathcal{Z}_2\cup\mathcal{Z}_4 = (\mathcal{D}\times\mathcal{X})\setminus\mathcal{Z}_\textrm{coast}$. Then for the system \eqref{eq:model}, any control law $u:\mathcal{D}\times\mathcal{X}\rightarrow\mathcal{U}$ satisfying \eqref{eq:mpc_conditions} and all of the following
	\begin{subequations}
		\begin{align}
%			u(t,\sigma,x) = 0 \iff x\in&\mathcal{X}_1(t) \cup \mathcal{X}_3(t) \,, \label{eq:cond_u0} \\
			%
			\hspace{-6pt}\psi_v(t+\Delta t, t, x) \leq 0, \;\;\;\;\;\;\;\;\;\;\;\;\;\;\;\;\;\;\;\;\;\;\;\;\;\, & \forall (t,\sigma,x)\in\mathcal{Z}_3 , \label{eq:cond_flow0} \\ 
			\hspace{-6pt}\psi_v(t+\Delta T, t, g(t, x, u(t,\sigma, x))) \leq 0, \;\; &\forall (t,\sigma,x)\in\mathcal{Z}_4 ,  \label{eq:cond_flow} \\ 
			\hspace{-6pt}V(t,g(t, x, u(t,\sigma, x))) \leq V(t,x), \;\;\;\;\;\, & \forall (t,\sigma,x)\in\mathcal{Z}_4 , \label{eq:cond_jump}
		\end{align}\label{eq:conditions}\end{subequations}%
	%where $y=g(t, x, u(t,\sigma, x))$, 
 will render the origin uniformly stable as in \cite[Def.~4.4]{Khalil}.
\end{theorem}

% The idea of this proof is to show that any of the conditions \eqref{eq:mpc_cond_flow_outside}, \eqref{eq:mpc_cond_jump_outside}, \eqref{eq:cond_flow0}, or \eqref{eq:cond_flow}-\eqref{eq:cond_jump} are sufficient to guarantee stability, so we can design a controller where different conditions apply

\begin{proof}
	Let $(t_k,\sigma(t_k))\in\mathcal{D}$ be an impulse opportunity, and let $(t_{k+1},\sigma(t_{k+1}))\in\mathcal{D}$ be the next impulse opportunity.
    For brevity, denote $z_k = (t_k, \sigma(t_k), x(t_k))$. 
    %Let $t_k$ be a sample time such that $(t_k,\sigma(t_k))\in\mathcal{D}$ is an impulse opportunity, and let $t_{k+1}$ be the next time of an impulse opportunity. 
    First, Lemma~\ref{lemma:one-step_mpc} implies that if $z_k\in\mathcal{Z}_1\cup\mathcal{Z}_2$, then $V(t_{k+1},x(t_{k+1})) \leq V(t_k, x(t_k))$. 
    
	Next, if $z_k\in\mathcal{Z}_3\cup\mathcal{Z}_4$, conditions \eqref{eq:cond_flow0}-\eqref{eq:cond_jump} similarly imply that $V(t_{k+1},x(t_{k+1})) \leq V(t_k, x(t_k))$. Specifically, if $z_k \in \mathcal{Z}_3 \subseteq\mathcal{Z}_\textrm{coast}$, then no impulse is applied, and \eqref{eq:cond_flow0} implies that $V(t,x(t))$ is nonincreasing along the flow $f$ for all $t\in[t_k,t_k+\Delta t)$ until the next impulse opportunity at $t_{k+1} = t_k + \Delta t$. Next, if $z_k\in\mathcal{Z}_4$, then a nonzero impulse is applied, \eqref{eq:cond_jump} implies that $V$ is nonincreasing during the impulse, and \eqref{eq:cond_flow} implies that $V(t,x(t))$ is nonincreasing along the flow $f$ for all $t\in(t_k,t_k+\Delta T)$ until the next impulse opportunity at $t_{k+1}=t_k+\Delta T$. Thus, $V(t_{k+1},x(t_{k+1})) \leq V(t_k,x(t_k))$ for all $(t_k,\sigma(t_k))\in\mathcal{D}$, so the origin is uniformly stable by the same argument as Lemma~\ref{lemma:one-step_mpc}.
\end{proof}

Compared to \cite{impulsive_event_triggered_stability,impulsive_iss,impulsive_event_triggered_iss}, Theorem~\ref{thm:stable} imposes stricter conditions on the flows \eqref{eq:cond_flow0}-\eqref{eq:cond_flow} in order to allow relaxed conditions on the jumps \eqref{eq:cond_jump} and the jump times. In \cite{impulsive_event_triggered_stability,impulsive_event_triggered_iss}, it is assumed that the flows are destabilizing and jumps are exponentially stabilizing, whereas Theorem~\ref{thm:stable} says that if we can restrict the flow \eqref{eq:cond_flow0}-\eqref{eq:cond_flow} to the set in \eqref{eq:decreasing_set} where $\dot{V}\leq 0$, as is often possible in practice, then the jump \eqref{eq:cond_jump} only needs to be stabilizing, not exponentially stabilizing. This coasting strategy can reduce control usage compared to {\color{edits}the exponentially stabilizing impulses in} \cite{impulsive_event_triggered_stability,impulsive_event_triggered_iss}, and is distinct from {\color{edits}the coasting strategy in} \cite{First_APF_Rendezvous} because of the explicit inclusion of a minimum time between impulses. Note that \eqref{eq:cond_flow0}-\eqref{eq:cond_flow} are identical to \eqref{eq:cbf_cond1}-\eqref{eq:cbf_cond2}, so a controller as in Theorem~\ref{thm:stable} will further render $\mathcal{S}_v$ in \eqref{eq:decreasing_set} forward invariant if $x(t_0)\in\mathcal{S}_v(t_0)$ and $\mathcal{Z}_3\cup\mathcal{Z}_4 = \mathcal{D}\times\mathcal{S}_v$. Finally, we present a result on asymptotic stability that we will use in Section~\ref{sec:simulations}.

\begin{corollary} \label{cor:asymptotic}
	Let the conditions of Theorem~\ref{thm:stable} hold. If there exists $\beta_1,\beta_2\in\mathcal{K}_r$ and $\Delta T_\textrm{max}\in\reals_{>0}$ such that 1)  \eqref{eq:scond_flow_outside}-\eqref{eq:scond_jump_outside} hold and 2) either 2a) \eqref{eq:scond_flow0}-\eqref{eq:scond_flow} hold or 2b) \eqref{eq:scond_jump}-\eqref{eq:scond_time} hold%
	% In the first case, you need to assume that both \eqref{eq:scond_jump_outside} and \eqref{eq:scond_flow_outside} hold, because if only \eqref{eq:scond_flow_outside} holds, then there is nothing to force the system to flow instead of jumping. Thus, both must hold, which means we do not need to place a minimum frequency on jumps
	\begin{subequations}
		\begin{align}
			&\hspace{-8pt}V(t+\Delta t, p(t+\Delta t, t, x)) - w \leq -\beta_2(w), \;\;\, \forall (\cdot)\in\mathcal{Z}_1 , \label{eq:scond_flow_outside} \hspace{-7pt} \\
			&\hspace{-8pt}V(t+\Delta T, p(t + \Delta T, t, y)) - w \leq -\beta_2(w), \, \forall (\cdot) \in\mathcal{Z}_2 , \label{eq:scond_jump_outside} \hspace{-7pt} \\
			&\hspace{-6pt}\psi_v(t+\Delta t, t, x) \leq -\beta_1(w), \;\;\;\;\;\;\;\;\;\;\;\;\;\;\;\;\;\;\;\;\;\;\;\;  \forall (\cdot) \in\mathcal{Z}_3 , \hspace{-7pt} \label{eq:scond_flow0}\\ 
			&\hspace{-6pt}\psi_v(t+\Delta T, t,y) \leq -\beta_1(w), \;\;\;\;\;\;\;\;\;\;\;\;\;\;\;\;\;\;\;\;\;\;\; \forall (\cdot)\in\mathcal{Z}_4 , \hspace{-7pt} \label{eq:scond_flow} \\ 
			&\hspace{-6pt}V(t,y) - w \leq -\beta_2(w), \;\;\;\;\;\;\;\;\;\;\;\;\;\;\;\;\;\;\;\;\;\;\;\;\;\;\;\;\;\; \forall (\cdot) \in\mathcal{Z}_4 , \hspace{-7pt} \label{eq:scond_jump}  \\
			&\hspace{-6pt}\,\sigma \geq \Delta T_\textrm{max} \implies u(t,\sigma,x) \neq 0, \;\;\;\;\;\;\;\;\;\;\; \forall (\cdot)\in\mathcal{D}\times\mathcal{X},  \label{eq:scond_time}
		\end{align}
	\end{subequations}
	where $(\cdot)= (t,\sigma,x)$, $y= g(t, x, u(t,\sigma, x))$, $w = V(t,x)$,
	then the origin is uniformly asymptotically stable \cite[Def.~4.4]{Khalil}.
\end{corollary}
\extendedversion{\begin{proof}
    The main idea of this proof is to show that there exists a convergent sequence $\{V_k\}_{k=1}^N$, where we denote $V_k = V(t_k,x(t_k))$, each $(t_k,\sigma(t_k))\in\mathcal{D}$ is an impulse opportunity, and $\Delta t \leq t_{k+1} - t_k \leq \Delta T_\textrm{max}$. We will do this in three parts. For brevity, denote $z_k = (t_k,\sigma(t_k),x(t_k))$.

    First, conditions \eqref{eq:scond_flow_outside}-\eqref{eq:scond_jump_outside} strengthen \eqref{eq:mpc_cond_flow_outside}-\eqref{eq:mpc_cond_jump_outside} so that the ``one-step MPC'' strategy is now asymptotically stabilizing for all $z_k\in\mathcal{Z}_1\cup\mathcal{Z}_2$. Specifically, if $z_k\in\mathcal{Z}_1$, then let $t_{k+1} = t_k + \Delta t$, so that \eqref{eq:scond_flow_outside} is equivalent to $V_{k+1} - V_k \leq -\beta_2(V_k) \leq -\beta_2(V_{k+1})$. Similarly, if $z_k \in \mathcal{Z}_2$, then \eqref{eq:scond_jump_outside} implies the same result for $t_{k+1} = t_k + \Delta T$.

    Second, if 2a holds, then impulses are stabilizing as in Theorem~\ref{thm:stable}, and flows are now asymptotically stabilizing for all $z_k\in\mathcal{Z}_3\cup\mathcal{Z}_4$. If $z_k\in\mathcal{Z}_3$, then let $t_{k+1} = t_k + \Delta t$. Then \eqref{eq:scond_flow0} implies that $v(t,x(t))\equiv\dot{V}(t,x(t)) \leq -\beta_1(V(t,x(t))) \leq -\beta_1(V_{k+1})$ for all $t\in(t_k,t_{k+1})$. It follows that $V_{k+1} - V_k \leq - \beta_1(V_{k+1}) \Delta t$. If 2a holds and instead $z_k\in\mathcal{Z}_4$, then let $t_{k+1} = t_k + \Delta T$. Then \eqref{eq:cond_jump} implies that $V$ is nonincreasing during the impulse, so \eqref{eq:scond_flow} similarly implies that $V_{k+1} - V_k \leq - \beta_1(V_{k+1}) \Delta T$.

    Third, if 2b holds, then flows are stabilizing as in Theorem~\ref{thm:stable}, and impulses are now asymptotically stabilizing for all $z_k\in\mathcal{Z}_3\cup\mathcal{Z}_4$. Condition \eqref{eq:scond_time} implies that impulses occur at least as frequently as $\Delta T_\textrm{max}$, so let $t_{k+1}$ be the time of the last impulse opportunity in $[t_k,t_k+\Delta T_\textrm{max}]$. Next, let $\{\tau_j\}_{j=1}^M$ be the sequence of impulse opportunity times starting at $t_k$ and ending at $t_{k+1}$. Then Theorem~\ref{thm:stable} implies that $V(\tau_{j+1},x(\tau_{j+1}))\leq V(\tau_j,x(\tau_j))$ for all $j\in\{1,\cdots,M\}$. Moreover, \eqref{eq:scond_jump}-\eqref{eq:scond_time} imply that there exists at least one $\tau_j \in \{\tau_j\}_{j=1}^{M-1}$ such that $V(\tau_{j+1},x(\tau_{j+1})) - V(\tau_j,x(\tau_j)) \leq -\beta_2(V(\tau_j,x(\tau_j)))$. It follows that $V_{k+1} - V_k \leq -\beta_2(V(\tau_j,x(\tau_j)))\leq -\beta_2(V_{k+1})$.

    Recall that $\mathcal{Z}_1\cup\mathcal{Z}_2\cup\mathcal{Z}_3\cup\mathcal{Z}_4 = \mathcal{D}\times\mathcal{X}$, so the combination of \eqref{eq:scond_flow_outside}-\eqref{eq:scond_jump_outside} and either \eqref{eq:scond_flow0}-\eqref{eq:scond_flow} or \eqref{eq:scond_jump}-\eqref{eq:scond_time} covers all possible states. In every case, we showed that $V_{k+1} - V_k \leq -\beta(V_{k+1})$ for some $\beta\in\mathcal{K}_r$. Equivalently, $V_{k+1} + \beta(V_{k+1}) \leq V_k$. Since each $V_k \geq 0$, this condition describes a convergent sequence $\{V_k\}_{k=1}^N$. If $\mathcal{T}$ is unbounded, then $N=\infty$, and $\lim_{k\rightarrow\infty}V_k = 0$. Note that this convergence is uniform in time, because $\beta$ is only a function of $V$ (i.e. $\beta$ is not a function of $t$ and $V$). Since \eqref{eq:model} is uniformly stable, $\|x\|$ satisfies $\|x\| \leq \alpha_1^{-1}(V(t,x(t)))$, the sequence $t_k$ satisfies $t_{k+1}-t_k \leq \Delta T_\textrm{max}$, and $V_k$ is uniformly convergent, it follows that the origin of \eqref{eq:model} is uniformly asymptotically stable.
\end{proof}}%
\regularversion{%\vspace{-8pt}%
\begin{proof}
    % The proof (see \cite[Cor.~1]{extended}) is a typical exercise in Lyapunov analysis and is omitted here for brevity.
    The proof is a standard exercise in Lyapunov analysis and is omitted here for brevity. See \cite[Cor.~1]{extended}.
    % The proof is based on standard Lyapunov analysis and is omitted here for brevity. See \cite[Cor.~1]{extended}.
\end{proof}}

That is, if the Lyapunov function $V$ is nonincreasing as in Theorem~\ref{thm:stable}, and either the flows \eqref{eq:scond_flow0}-\eqref{eq:scond_flow} or the jumps \eqref{eq:scond_jump}-\eqref{eq:scond_time} cause $V$ to strictly decrease, then the origin is asymptotically stable. 
Again, we provide alternative ``one-step MPC'' conditions \eqref{eq:scond_flow_outside}-\eqref{eq:scond_jump_outside} in case \eqref{eq:scond_flow0}-\eqref{eq:scond_time} cannot be satisfied because $x(t)\notin\mathcal{S}_v(t)$. 
If we further assume that $\beta_1$ and $\beta_2$ are linear functions, then the conditions in Corollary~\ref{cor:asymptotic} become special cases of \cite[Thm.~1]{impulsive_iss}. % wherein both flows and jumps are stabilizing

\subsection{Examples of Bounding Functions} \label{sec:examples}

In this subsection, we discuss in more detail how to develop $\psi_h$ and $\psi_v$ to use in the preceding theorems.
Suppose second order dynamics such that $x = [r\transpose,\,\dot{r}\transpose]\transpose\in\reals^n$ for flow dynamics $\ddot{r} = f_r(x)$. First, an obstacle avoidance constraint can be written using the following form of CBF $h$ \cite{Sreenath2016}:
\begin{subequations}\begin{align}
    &\kappa(t,x) = \rho - \| r - r_0(t) \| \label{eq:def_kappa} \\
    &h(t,x) = \kappa(t,x) + \gamma \dot{\kappa}(t,x) \label{eq:cbf_obstacle} \\
    &\psi_h(t+\delta, t, x) = \max\big\{ h(t,x), \nonumber \\ &\;\;\;\;\;\;\;\;\;\;\;\kappa(t,x) + (\gamma + \delta) \dot{\kappa}(t,x) + \left( {\textstyle\frac{1}{2}}\delta^2 + \gamma \delta \right) \ddot{\kappa}_\textrm{max} \big\} \label{eq:psi_h} %\\
    %
%    &\ddot{\kappa}_\textrm{max} = \max_{\substack{s\in[t,t+\delta] \\ y \in \{ y \mid \exists \tau \in [t,t+\delta] : V(\tau,y) \leq V(t,x) \} } } \ddot{\kappa}(s,y)
\end{align}\end{subequations}
where {\color{edits}$\rho\in\reals_{>0}$ is the obstacle radius}, $\gamma\in\reals_{>0}$ {\color{edits}is a constant}, $r_0:\mathcal{T}\rightarrow\reals^{n/2}$ is the center of {\color{edits}the} obstacle, and $\ddot{\kappa}_\textrm{max}\in\reals_{\geq0}$ is an upper bound on the possible values of $\ddot{\kappa}$ between $t$ and $t+\delta$. 
%This bound can be local (i.e. specific to $t$ and $x(t)$) or global (constant for all $t$ and $x(t)$) \cite{LCSS}. 
We use formula \eqref{eq:psi_h} for the bound $\psi_h$ because $\kappa$ in \eqref{eq:def_kappa} is not thrice differentiable, so we cannot make use of any higher order derivatives.
Next, the rate of change of a \extendedversion{quadratic} Lyapunov function $V(t,x) \extendedversion{= x\transpose P x}$ can be upper bounded as
\regularversion{

$\;$ \vspace*{-16pt}}
% \begin{subequations}%
\begin{align}
    % &\hspace{-4pt}V(t,x) = x\transpose P x \label{eq:V_quad} \\
    %
    &\hspace{-7pt}\psi_v(t+\delta, t, x) \hspace{-1pt}=\hspace{-1pt} \dot{V}(t,x) \hspace{-0.5pt}+\hspace{-0.5pt} \max \{0, \ddot{V}(t,x) \} \delta \hspace{-0.5pt}+\hspace{-0.5pt} {\textstyle\frac{1}{2}} \dddot{V}_\textrm{max} \delta^2 \hspace{-6pt} \label{eq:psi_v} % \\
    %
%    &\dddot{V}_\textrm{max}(t,x) = \max_{\substack{s\in[t,t+\delta] \\ y \in \{ y \mid \exists \tau \in [t,t+\delta] : V(\tau,y) \leq V(t,x)} } \dddot{V}(t,x)
\end{align}%
% \end{subequations}%
where $\dddot{V}_\textrm{max}\in\reals_{\geq0}$ is an upper bound on the \extendedversion{possible }values of $\dddot{V}$. 
\extendedversion{We stop the approximation $\psi_v$ at the third derivative of $V$, because we note that $\dddot{V}$ is a function of only derivatives and higher powers of $f_r$, so higher order approximations do not substantially decrease conservatism.}
% To justify stopping the approximation $\psi_v$ at the third derivative of $V$, note that for the dynamics $\ddot{r} = f(x)$, the formula for $\ddot{V}$ is a function of $\ddot{r}$ and $\dddot{r}$, which are both terms in which $f(x)$ will appear explicitly. By contrast, $\dddot{V}$ is only a function of derivatives and higher powers of $f$, and is therefore often much smaller than $\ddot{V}$. Thus, extending the approximation encoded in $\psi_v$ beyond the third derivative of $V$ does not yield much benefit. 
%
% To see this, consider the derivatives of $V$ for the double integrator system
% \begin{align}
% 	&V(x) \hspace{-1pt}=\hspace{-1pt} x\transpose \hspace{-1pt}\begin{bmatrix} p_{11} & p_{12} \\ p_{12} & p_{22} \end{bmatrix}\hspace{-1pt} x, \;\; \dot{x} \hspace{-1pt}=\hspace{-1pt}\hspace{-1pt} \begin{bmatrix} 0 & 1 \\ 0 & 0 \end{bmatrix} x, \\
% 	%  
% 	&\dot{V}(x) \hspace{-1pt}=\hspace{-1pt} x\transpose \hspace{-1pt}\begin{bmatrix} 0 & p_{11} \\ p_{11} & 2 p_{12} \end{bmatrix}\hspace{-1pt} x, \;
% 	\ddot{V}(x) \hspace{-1pt}=\hspace{-1pt} x\transpose \hspace{-1pt}\begin{bmatrix} 0 & 0 \\ 0 & 2 p_{11} \end{bmatrix}\hspace{-1pt} x, \;
% 	\dddot{V}(x) \hspace{-1pt}=\hspace{-1pt} 0 \,. \nonumber
% \end{align}

\subsubsection{Decreasing Conservatism} \label{sec:mpc}

Note that the upper bounds derived in \cite{LCSS,AIAA,CBFs_Mechanical,sampled_data_2022,sampled_data_unknown,self_triggered_cbf} and implemented above were intended for relatively short horizon times $\tau - t$. For very large horizon times, these upper bounds can become overly conservative. We can optionally decrease this conservatism by breaking the interval $\tau - t$ into $n_\psi\in\mathbb{N}$ smaller intervals. To this end, let $\delta = (\tau-t)/n_\psi$ and $\tau_j = t + j \delta$, and for a scalar function $h:\mathcal{T}\times\mathcal{X}\rightarrow\reals$, replace $\psi_h$ as above with $\psi_h^*:\mathcal{T}\times\mathcal{X}\rightarrow\reals^{n_\psi}$ with elements defined as
\begin{equation}
	[\psi_h^*(\tau,t,x)]_j = \psi_h(\tau_j,\tau_{j-1},p(\tau_{j-1},t,x)) \label{eq:mpc}
\end{equation}
for $j = 1, \cdots, n_\psi$. That is, $\psi_h^*$ makes $n_\psi$ exact state predictions using $p$ in \eqref{eq:state_flow}, which could be more expensive to compute, and bounds the evolution between these predictions using the original $\psi_h$ function. This division is analogous to MPC with a control horizon of 1, a prediction horizon of $n_\psi$, {\color{edits}and a discretization margin encoded in $\psi_h$}. In the above work, all statements of the form $\psi_a(\cdot) \leq 0$, where $a$ is $h$ or $v$, can be equivalently replaced by $\psi_a^*(\cdot) \leq 0$ elementwise. We will demonstrate the utility of this strategy \extendedversion{in simulation }in Section~\ref{sec:simulations}.

\section{Simulations} \label{sec:simulations}

We validate the above methods by simulating an impulsive system representative of spacecraft docking in low Earth orbit. Let 
%\color{edits}$r,\dot{r} \in \reals^2$, $x = [r\transpose,\, \dot{r}\transpose]\transpose \in\color{black} 
$\mathcal{X} = \reals^4$\color{black}, $\mathcal{U} = \reals^2$, let $\color{edits}\mu\color{black}\in\reals_{>0}$ be constant, and let
%\begin{subequations}
% \begin{equation}
%     f(\cdot) \hspace{-1pt}=\hspace{-1pt} \begin{bmatrix} 0 & 0 & 1 & 0 \\ 0 & 0 & 0 & 1 \\ 3 n^2 & 0 & 0 & 2 n \\ 0 & 0 & -2 n & 0 \end{bmatrix}\hspace{-2pt} x ,\; g(\cdot) \hspace{-1pt}=\hspace{-1pt} x + \begin{bmatrix} 0 & 0 \\ 0 & 0 \\ 1 & 0 \\ 0 & 1 \end{bmatrix}\hspace{-2pt} u \,.\label{eq:docking}
% \end{equation}%
% \begin{equation}
%     \color{edits}\ddot{r} = -\frac{\mu r}{\|r\|^3} \,. \label{eq:docking}
% \end{equation}
\begin{equation}
    \color{edits}f(\cdot) = \begin{bmatrix} x_3 \\ x_4 \\ -\mu x_1 / (x_1^2 + x_2^2)^{3/2} \\ -\mu x_2/(x_1^2 + x_2^2)^{3/2} \end{bmatrix}, \; g(\cdot) = \begin{bmatrix} x_1 \\ x_2 \\ x_3 + u_1 \\ x_4 + u_2 \end{bmatrix} \,. \label{eq:docking}
\end{equation}
%\end{subequations}%
Let there be four CBFs $h_i$ of the form \eqref{eq:cbf_obstacle} for various obstacles $r_i(t)\color{edits}\in\reals^2$, \extendedversion{\color{edits}representing other objects in orbit,} 
with $\psi_{h_i}$ as in \eqref{eq:psi_h}. Let there be an additional constraint $\kappa_5(t,x) = \color{edits}(r - r_5)\transpose (\dot{r}_5/\|\dot{r}_5\|) \color{black}\leq 0$ with associated CBF $h_5$ also as in \eqref{eq:cbf_obstacle}. {\color{edits}That is, $\kappa_5$ encodes that the controlled satellite $r$ must always lie behind an uncontrolled target satellite $r_5(t)\in\reals^2$. Let $x_t(t) = [r_5(t)\transpose,\,\dot{r}_5(t)\transpose]\transpose$.} We choose a Lyapunov function $V(t,x) = \color{edits}(x-x_t(t))\transpose P (x-x_t(t))$ and approximation $\psi_v^*$ as in \eqref{eq:psi_v} and \eqref{eq:mpc}. Let $\gamma_1,\gamma_2\in\reals_{\geq0}$ and $J\in\reals_{>0}$ be constants. The chosen control law is
\begin{subequations}\begin{equation}
    u = \begin{cases} 
        0 &  \psi_v(\cdot) \leq \gamma_1 V(t,x)  \textrm{ and } \psi_{h_i}(\cdot) \leq 0,\,i\in\mathcal{I} \\ 
        u^* & \textrm{else} 
    \end{cases}
\end{equation}
where $(\cdot) = (t+\Delta t, t, x)$, $\mathcal{I}=\{1,2,3,4,5\}$, and $u^*$ is
\begin{align}
    u^* &= \argmin_{u\in\reals^2} u\transpose u + J d^2 \\
    &\;\;\; \textrm{s.t. } \psi_v^*(t+\Delta T, t, g(t, x, u)) \leq \gamma_1 V(t,x) + d \\ 
        &\;\;\;\;\;\;\;\;\; V(t,g(t,x,u)) \leq \gamma_2 V(t,x) + d \\
        &\;\;\;\;\;\;\;\;\; \psi_{h_i}(t+\Delta T, t, g(t,x,u)) \leq 0, \, i\in\mathcal{I} \label{eq:control_h} \,.
\end{align}\label{eq:docking_control}\end{subequations}
We assume that the optimization \eqref{eq:docking_control} is always feasible, though we note that this is difficult to guarantee when there are multiple CBFs \cite{Sharing_CBFs,compatibility_checking,ACC2023}. We simulated \eqref{eq:docking_control} using {\color{edits}various choices of} $\Delta T$, and then repeated these simulations with $\psi_h$ in \eqref{eq:control_h} replaced with $\psi_h^*$ as in \eqref{eq:mpc} with $n_{\psi_h} = 10$. {\color{edits}The resultant trajectories, converted to Hill's frame for visualization, are shown in Fig.~\ref{fig:trajectory}, and full results are shown in the video below\footnote{\url{https://youtu.be/_o-FAGbvfgg}}. A comparison to a trajectory pre-planner is also shown in Fig.~\ref{fig:trajectory}, and details on select trajectories are shown in Figs.~\ref{fig:control}-\ref{fig:lyap}}. All simulation code and parameters can {\color{edits}also} be found below\footnote{\url{https://github.com/jbreeden-um/phd-code/tree/main/2023/LCSS\%20Impulsive\%20Control}}.

%In Figs.~{}, we show simulations using the constraint indices $\mathcal{I}_4$ and $\mathcal{I}_5$, using $\Delta T$ of both 15 seconds and 30 seconds, both with $\psi_h$ as in \eqref{eq:obstacle} and using $\psi_h^*$ as in \eqref{eq:mpc} with $n_h = 10$. We also compare to an offline precomputed solution. Full simulation code and parameters can be found below\footnote{\todo{link}}.

All of the simulations {\color{edits}in Fig.~\ref{fig:trajectory}} remained safe{\color{edits}, and eight of the nine trajectories converged to the target. The trajectory using $\psi_h$ with $\Delta T = 60$ was so conservative that it immediately turned away from the target, whereas trajectories using $\psi_h^*$ still converge with much larger $\Delta T$, though the rate of convergence is slow for $\Delta T \geq 420$. This is} because $\psi_h^*$ implements \eqref{eq:psi_h} with a smaller, less conservative, $\delta$ than $\psi_h$ alone. That said, this decreased conservatism came at an average computational cost per control cycle{\color{edits}, for $\Delta T = 45$, of 0.22~s using $\psi_h^*$ and 0.022~s using $\psi_h$\regularversion{, both run on a 3.5~GHz CPU}.}\extendedversion{. These computation times are for a 3.5~GHz CPU, and would likely be much larger onboard a spacecraft processor.}
{\color{edits}The total fuel consumption varied from 188~m/s ($\Delta T = 30$ with $\psi_h$) to 18.2~m/s ($\Delta T = 300$ with $\psi_h^*$). For comparison, the pre-planned trajectory consumed between 12.2~m/s and 13.9~m/s depending on the choice of $\Delta T$. This improvement is expected since \eqref{eq:docking_control} only considers $T$ seconds of the trajectory at a time, whereas a pre-planner can optimize over longer sequences.}

\begin{figure}
	\centering
	\includegraphics[width=0.9\columnwidth]{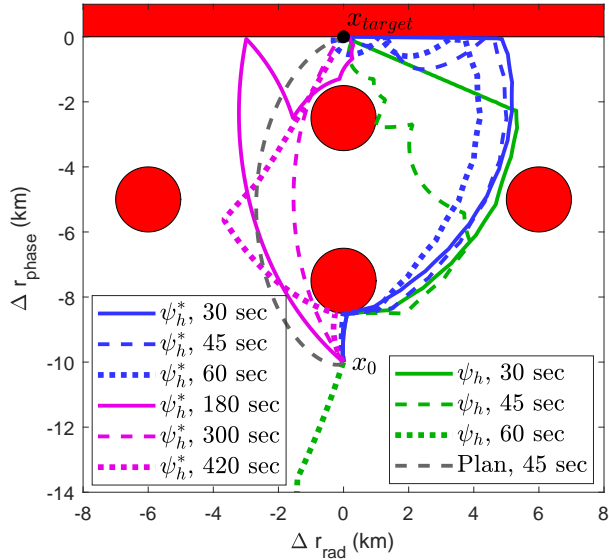}
	\caption{Trajectories of \eqref{eq:model} and \eqref{eq:docking} subject to the control \eqref{eq:docking_control}}
	\label{fig:trajectory}
    % \vspace{-6pt}
\end{figure}

\begin{figure}
	\centering
	\includegraphics[width=\columnwidth,trim={0in, 0.34in, 0in, 0in},clip]{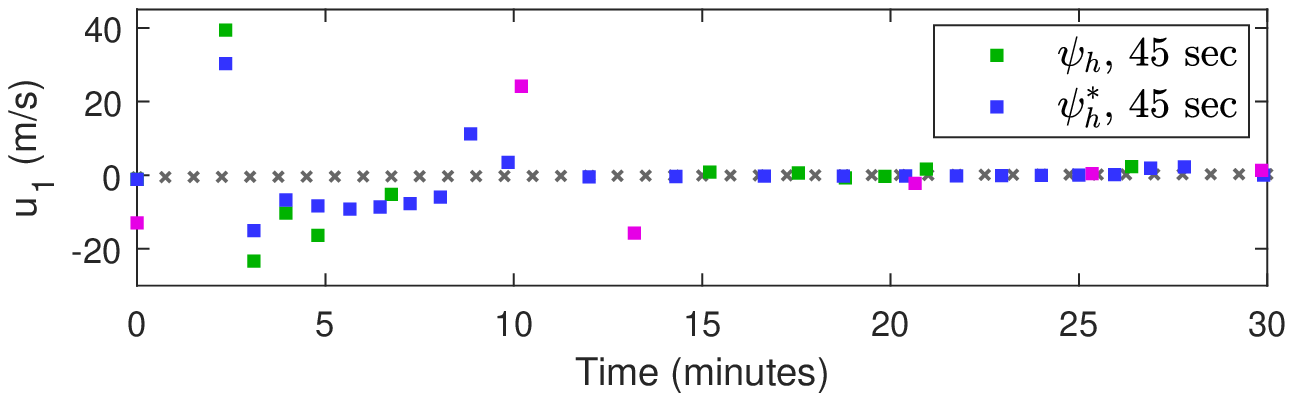}
	\includegraphics[width=\columnwidth,trim={0in, 0in, 0in, 0in}]{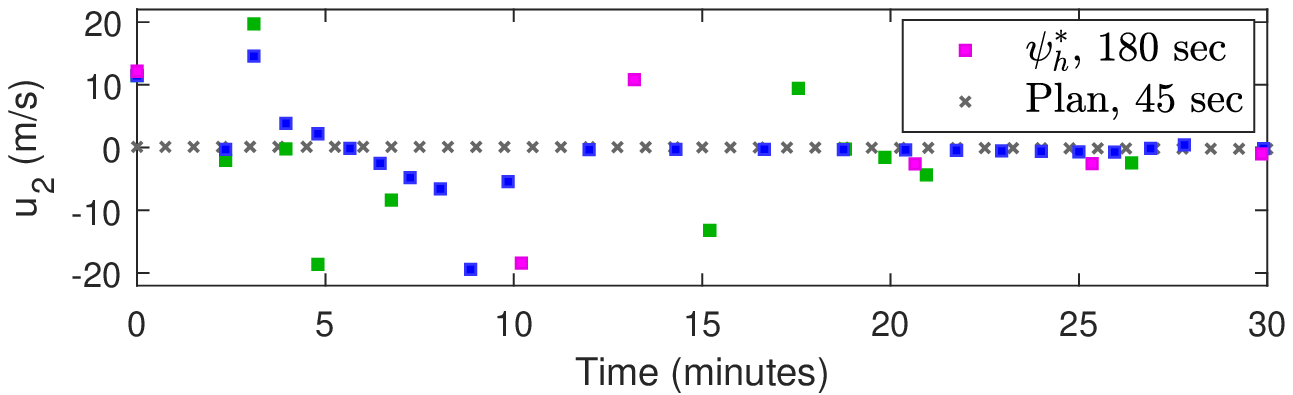}
	\caption{Control inputs along selected trajectories in Fig.~\ref{fig:trajectory}}
	\label{fig:control}
\end{figure}
% \vspace{6pt}
\begin{figure}
	\centering
	\includegraphics[width=\columnwidth]{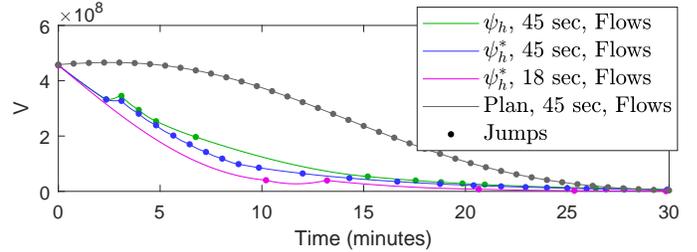}
	\caption{Lyapunov function along selected trajectories in Fig.~\ref{fig:trajectory}}
	\label{fig:lyap}
\end{figure}

\section{Conclusions} \label{sec:conclusions}

We have developed a methodology for extending the provable set invariance guarantees provided by CBFs to systems with impulsive actuators subject to a minimum dwell time constraint, and for ensuring asymptotic stability in the same systems. We \extendedversion{then }encoded the resulting conditions in an optimization-based control law, which was successful in a simulated spacecraft docking. The conditions presented are generally nonlinear in the control input, thus leading to controllers that are \regularversion{\color{edits}nonlinear programs}\extendedversion{solutions to nonlinear optimizations}. We showed how one can reduce the conservatism of these controllers, in exchange for greater computational cost, by dividing the safety prediction horizon into multiple intervals using an MPC-like strategy.
%We then established stability criteria that can be used as part of the same nonlinear optimization already being used to ensure safety. The use of ``one-step'' MPC constraints reduced conservatism in both the safety and stability conditions, as was demonstrated in the spacecraft docking simulation. 
Future research directions might consider extensions {\color{edits}to systems with disturbances,} %of these minimum dwell time results, 
methods to further decrease conservatism, or the use of ITCBFs with optimal trajectory planning.

\bibliographystyle{ieeetran}
\bibliography{sources}

\end{document}